\documentclass[12pt]{article}
\usepackage[utf8]{inputenc}
\usepackage{color}
\usepackage[dvipsnames]{xcolor}
\usepackage{amsmath}
\usepackage{parskip}
\usepackage[showframe=false]{geometry}
\usepackage{changepage}
\usepackage{graphicx}
\usepackage{tikz}
\usepackage{ mathdots }
\usetikzlibrary{arrows.meta}
\usetikzlibrary{patterns}
\usepackage{ dsfont }
\usepackage{diagbox}
\usepackage[refpage]{nomencl}
\usepackage{amsthm}
\usepackage[compact]{titlesec}         
\titlespacing{\section}{10pt}{10pt}{10pt} 
\AtBeginDocument{
  \setlength\abovedisplayskip{10pt}
  \setlength\belowdisplayskip{10pt}}
\usepackage{tikz}
\usetikzlibrary{positioning}
\usepackage{amssymb}
\usepackage{tikz}
\usetikzlibrary{patterns}
\usetikzlibrary{decorations.pathmorphing}
\usetikzlibrary{angles,quotes}
\tikzset{snake it/.style={decorate, decoration=snake}}
\usepackage{cancel}
\usepackage{caption}
\usepackage{subcaption}
\usepackage{ textcomp }
\makeindex
\usepackage{ gensymb }
\usepackage{float}
\usepackage{rotating}
\usepackage{tabularx}
\usepackage{delarray}

\newcommand{\tpmod}[1]{{\@displayfalse\pmod{#1}}}

\newcommand{\Sym}[1]{\mathrm{Sym}({#1})}

\usepackage{amsmath,xcolor}

\usepackage{mathtools}
\usepackage[shortlabels]{enumitem}

\usepackage[english]{babel}
\usepackage{soul}
\newtheorem{theorem}{Theorem}[section]

\newtheorem{lemma}[theorem]{Lemma}
\newtheorem{definition}[theorem]{Definition}
\newtheorem{proposition}[theorem]{Proposition}

\theoremstyle{definition}

\usepackage{tikz}
\title{\bf{The Bruhat Order of a Finite Coxeter Group and Elnitsky Tilings}}
\author{Rob Nicolaides and Peter Rowley}

\begin{document}

\maketitle
\begin{abstract} Suppose that $W$ is a finite Coxeter group and $W_J$ a standard parabolic subgroup of $W$.  The main result proved here is that for any for any $w \in W$ and reduced expression of $w$ there is an Elnitsky tiling of a $2m$-polygon, where $m = [W : W_J]$.  The proof is constructive and draws together the work on E-embedding in \cite{nicolaidesrowley1} and the deletion order in \cite{nicolaidesrowley3}. Computer programs which produce such tilings may be downloaded from \cite{github} and here we also present examples of the tilings for, among other Coxeter groups,  the exceptional Coxeter group $\mathrm{E}_8$.

\end{abstract}

\section{Introduction}\label{Section introduction}

In his PhD thesis, \cite{elnitskyThesis} (see also \cite{elnitsky}), Elnitsky examined a fascinating relationship between the reduced words of a selection of Coxeter groups and certain tilings of polygons. More specifically, for each of the families of Coxeter groups of types A, B and D, he produces an elegant bijection between classes of their reduced words and particular tilings of polygons.  These, now referred to as Elnitsky tilings,  were produced using the degree $n$ permutation representation for the Coxeter group $\mathrm{A}_{n-1}$ and the signed permutation representations for types B and D. The methods revealed in this paper make use of any of the faithful permutation representations on the cosets of standard parabolic subgroups of any finite Coxeter group, so yielding an extensive collection of tilings,  with the tilings in \cite{elnitskyThesis}  as a special case. Our main result is as follows.

\begin{theorem}\label{maintheorem} Suppose that $W$ is a finite Coxeter group with $W_J$ a standard parabolic subgroup of $W$ which is core-free. Set $m = [W : W_J]$. Then for any $w \in W$ and reduced expression of $w$ there is an Elnitsky tiling of a $2m$-polygon.

\end{theorem}

Theorem 1.2 of   \cite{nicolaidesrowley1}, shows how certain embeddings (that is, injective homomorphisms) of Coxeter groups into symmetric groups naturally give rise to bijections between classes of reduced words and certain tilings of polygons. These embeddings, called E-embeddings with the E standing for Elnitsky, were introduced in  \cite{nicolaidesrowley1} and are defined as follows.

\begin{definition}\label{def e-embedding} Suppose that $W$ is a Coxeter group and that $\varphi : W \hookrightarrow \Sym{n}$ is an embedding. Then $\varphi$ is an E-embedding for $W$ if for all $u,v \in W,$
$$ u <_R v \text{ implies } \varphi(u) <_B \varphi(v).$$
\end{definition}

Here, $<_R$ and $<_B$ denote the weak right order and the Bruhat order respectively, their full definitions are recalled in Section \ref{Section background}.  We use $\Sym{n}$ to denote the symmetric group of degree $n$ which we regard as a Coxeter group with its fundamental reflections being $\{(i , i+1) \; | \; i = 1, \dots, n-1 \}$. Further we shall write our permutations on the right of the elements in $\{1, \dots, n \}$ and so composition is evaluated from left to right. For a Coxeter group $W$, $S$ will denote its set of fundamental (or simple) reflections.

No means for constructing E-embeddings were provided in \cite{nicolaidesrowley1} and so the focus of this article is to provide a universal construction for all finite Coxeter groups. This is done by exploring a special subset of the E-embeddings which we call \textit{strong E-embeddings}, and will enable us to give explicit constructions of the tilings.

\begin{definition}[Strong E-embedding]\label{Definition Strong E-embedding} Suppose that $W$ is a Coxeter group. 
Let $\varphi: W \hookrightarrow \Sym{n}$ be an embedding for some positive integer $n$. Then $\varphi$ is a {strong E-embedding} if
for all $w \in W$ and $s \in S$, $$w <_R ws \text{\quad implies \quad} \varphi(w) <_B \varphi(w)t_i$$ for each $i = 1 \ldots k$, where $\varphi(s) = t_1\ldots t_k$ is such that each $t_i$ is a transposition in $\Sym{n}$ and $t_i, t_j$ are disjoint for all $1 \le i < j \le k$.
\end{definition}

As we will see in Theorems \ref{Theorem Bruhat Refining iff see} and \ref{Theorem Bruhat Equivelence}, in order to construct strong E-embedding,  we need a total order of the Coxeter group $W$ which is a refinement of the Bruhat order on $W$.  There are many possible choices for total orders which are refinements of the Bruhat order. A particularly easy type of example may be made by first ordering by element length and among those of a fixed length then dictating an arbitary total order. The last part of this recipe is, to say the least, rather unsatisfactory given the richness of Coxeter groups  particularly in connection with the Bruhat order. We prefer to employ the deletion order,  which is canonical, is a refinement of the Bruhat order and also has many interesting properties (see \cite{nicolaidesrowley3} for more on this order). One particularly noteworthy feature is that (except in small cases) the deletion order is not ranked by element length. Moreover the tilings obtained using the deletion order align with the Elnitsky tilings in  \cite{elnitsky} for the Coxeter groups of type A, B and D.  In Section \ref{Section Constructing Elnitsky's Tilings} where we present examples of tilings for groups of type $\mathrm{D}_5, \mathrm{A}_2 \times \mathrm{A}_3$ band $\mathrm{E}_8$,  we shall employ the deletion order.  These examples were produced with programs which are available from \cite{github} so as the reader can produce tilings for other elements of other Coxeter groups.

In Section \ref{Section background} we review some standard results we will need about Coxeter groups, focussing on the Bruhat order and minimal coset representatives for standard parabolic subgroups. Then Section \ref{Section Main Results} is devoted to proving all our main results, starting off by showing that strong E-embeddings are indeed E-embeddings. We then show that the definition can be reformulated as follows.

\begin{proposition}\label{Proposition Strong E-embeddings Weak Order Bruhat Equivalence} Suppose that $W$ is a Coxeter group with $T$ its set of reflections.
    Let $\varphi: W \hookrightarrow \Sym{n}$ be an embedding for some positive integer $n$. Then the following are equivalent:
    \begin{enumerate}[(i)]
        \item $\varphi$ is a strong E-embedding;
        \item for all $w \in W$ and for all $t \in T$, $$w <_B wt \text{\quad implies \quad} \varphi(w) <_B \varphi(w)t_i$$ for each $i = 1 \ldots k$, where $\varphi(t) = t_1\ldots t_k$ is such that each $t_i$ is a transposition in $\Sym{n}$ and $t_i, t_j$ are disjoint for all $1 \le i < j \le k$.
    \end{enumerate}
\end{proposition}

The remainder of Section \ref{Section Main Results} aims to construct strong E-embeddings from Cayley embeddings equipped with a total ordering. We will use $\ll$ to denote a total ordering on $W$ meaning for all distinct $u,v \in W$ either $u \ll v$ or $v \ll u$. We define $L_\ll(w) = |\{u \in W \,|\, u \ll w\}|$. Then the Cayley embedding induced from $\ll$ is given by $\varphi_\ll: W \hookrightarrow \Sym{|W|}$ where $$L_\ll(u)\varphi_\ll(w) = L_\ll(uw^{-1}) \; \mbox{for all}  \; u,v \in W.$$ 

This means that 
$$\varphi_\ll(s) = \prod_{w <_R ws} (L_\ll(w), L_\ll(ws))$$ for each $s \in S$.

We ask for which total orderings $\ll$, does $\varphi_\ll$ happen to be a strong E-embedding? The answer is given by our next theorem.

\begin{theorem}\label{Theorem Bruhat Refining iff see}
Let $\ll$ be a total ordering of $W$ such that $L_\ll(1)=0$. Then the following are equivalent:
\begin{enumerate}[(i)]
    \item $\ll$ is a refinement of the Bruhat order;
    \item $\varphi_\ll$ is a strong E-embedding.
\end{enumerate}
\end{theorem}

For $\ll$ to be a refinement of the Bruhat order, it needs to satisfy the condition that $u <_B v$ implies $ u \ll v$ for all distinct $u,v \in W$.  Theorem \ref{Theorem Bruhat Equivelence} gives a more general result than Theorem \ref{Theorem Bruhat Refining iff see}. It relaxes the condition that $\varphi_\ll$ needs to be a Cayley embedding where instead, we can take any $J\subseteq S$ and consider permutations restricted to the minimal coset representatives, $W^J$, of $W_J = \,<J>$ (precise details are given in Section \ref{Section Main Results}).

\begin{theorem}\label{Theorem Bruhat Equivelence}
Let $J \subseteq S$ and $\ll^J$ be a total ordering of $W^J$ such that $L_{\ll^J}(1) = 0$. Then the following are equivalent:
    \begin{enumerate}[(i)]
        \item $\varphi_\ll^J$ is a strong E-embedding;
        \item $\ll^J$ is a total refinement of  $<_B$, restricted to $W^J$.
    \end{enumerate}
\end{theorem}

For further material on tilings the reader may consult  \cite{epty},  \cite{hamanaka}, \cite{nicolaidesrowley2} \cite{tenner1} and \cite{tenner2}.
 
\section{Background}\label{Section background}

A group, $W$, is a \emph{(finitely generated) Coxeter group} if it admits a presentation over some finite set $S = \{s_1,\ldots,s_n\}$, satisfying $$ W = \langle \, S \,\, | \,\, (s_is_j)^{m_{ij}} = 1 \,\rangle$$ Here, for each $i \ne j$, $m_{ij}$ is a positive integer such that $m_{ij} = m_{ji} > 1$ and $m_{ii} = 1$. We call the pair $(W, S)$ a \emph{Coxeter system} and the elements of $S$ the fundamental reflections of $W$. The set of \emph{reflections} of a Coxeter system is the set $T:= S^W = \{ ws_iw^{-1} \; |  \; \ w \in W, s_i \in S \}$.

Let $\Gamma(W)$ be the graph whose nodes are $S$ and whose set of edges is $\{\{s_i,s_j\} \, | \, m_{ij} \ne 2\}$. Sometimes $\Gamma(W)$ is called the Dynkin diagram of $W$. A Coxeter group is called \emph{irreducible} if $\Gamma(W)$ is connected. For the remainder of this paper, we assume the Coxeter groups we consider are finite. 

An example, as already mentioned,  of a Coxeter group is the symmetric group.  Take $S = \{s_1,\ldots,s_n\}$ and consider $$ W = \langle \, S \,\, | \,\, (s_is_j)^{m_{ij}} = 1 \,\rangle$$ such that $m_{i j} = 3$ if $|i-j|=1$ and $m_{ij} = 2$ otherwise for $i \ne j$. The map sending $s_i$ to the adjacent transposition $(i,i+1) \in \Sym{n+1}$, extends uniquely to an isomorphism from $W$ to $\Sym{n+1}$. 

For a general Coxeter group, $W$, and $w \in W$, a \emph{word} for $w$ over $S$ is an expressions of the form $w = s_{i_1}s_{i_2}\ldots s_{i_k}$ where $s_{i_j} \in S$ for each $j = 1,\ldots,k$. A word is called \emph{reduced} if it is of minimal length and we denote that length by $\ell(w)$.

Two important posets related to the reduced words of Coxeter groups are the \emph{weak (right) order}, denoted $<_R$, and the \emph{Bruhat order}, denoted $<_B$. The weak order is defined as the reflexive, transitive closure of the covering relations $w \triangleleft_R ws$ if and only if $\ell(ws) = \ell(w) + 1$ where $w \in W$ and some $s \in S$.
Similarly,  the Bruhat order is defined by the covering relations $w \triangleleft_B wt$ if and only if $\ell(wt) = \ell(w) + 1$, where $w \in W$ and some $t \in T$. 
The weak order can be characterized in terms of prefixes of reduced words (Proposition 3.1.2 of \cite{cocg}): for all $u,v \in W$, $u <_R v$ if and only if there is a reduced word for $u$ that is a prefix of some reduced word for $v$.
 Less obviously, the Bruhat order can also be characterized in terms of subwords of reduced words as stated in Lemma \ref{Lemma Bruhat Subwords}.

\begin{lemma}\label{Lemma Bruhat Subwords}
    For $u,v \in W$, the following are equivalent:
    \begin{enumerate}[$(i)$]
        \item $u <_B w$;
        \item for every reduced word of $w$, there exists a subword that is a reduced word for $u$; and
        \item some reduced word of $w$ has a subword that is a reduced word for $u$.
    \end{enumerate}
\end{lemma}
\begin{proof}
    This is Corollary 2.2.3 of \cite{cocg}.
\end{proof}

For the symmetric group, the Bruhat order has the following description in terms of its usual group action.

\begin{lemma}\label{Lemma Bruhat in Symmetric group}
    Let $w, t\in \Sym{n}$ for some positive integer $n$ such that $t$ is a transposition $t = (a, b)$ with $1 \le a < b \le n$. Then $$ w <_B wt $$ if and only if $(a)w^{-1} < (b)w^{-1}$.
\end{lemma}
\begin{proof}
        This is a direct consequence of Lemma 2.1.4 of \cite{cocg} and taking into account that we multiply permutations from left to right.
\end{proof}

 \begin{lemma}\label{Lemma Bruhat t sandwich}
Let $u,v \in W$ and $t \in T$ be such that $u <_B ut$ and $v <_B tv$, then $uv <_B utv$.
\end{lemma}
\begin{proof}
    This is Lemma 2.2.10 of \cite{cocg}.
\end{proof}

Let $J \subseteq S$, then we define $W_J = \langle J \rangle$. These subgroups are known as the  \emph{standard parabolic subgroups}. Each left (or right) coset of $W_J$ has a unique element of minimal length. The set of these minimal left coset representatives is denoted by $W^J$.
This leads to a unique decomposition property.

\begin{lemma}\label{Lemma parabolic unique decomposition}
    For all $J \subseteq S$ and $w \in W$, there exists unique $w^J \in W^J$, $w_J \in W_J$ such that 
    $w = w^Jw_J$. Moreover, $\ell(w) = \ell(w^J)+\ell(w_J)$.
\end{lemma}
\begin{proof}
    This is Proposition 2.2.4 of \cite{cocg}.
\end{proof}

If $w \in W$, we will use $w^J$ and $w_J$ to denote its unique factors as given in Lemma \ref{Lemma parabolic unique decomposition} from $W^J$ and $W_J$, respectively.

\begin{lemma}\label{Lemma Parabolic Projection preserves Bruhat}
Let $J \subseteq S$ and $u,v \in W$. If $u <_B v$, then $u^J <_B v^J$.
\end{lemma}
\begin{proof}
    This is Lemma 2.5.1 of \cite{cocg}.
\end{proof}

We note that if $I\subseteq J \subseteq S$, then $W^J \subseteq W^I$. To see this, note that $W_I \le W_J$ and so each coset of $W_J$ is partitioned by those of $W_I$ by Lagrange's theorem. Since each coset of a standard parabolic subgroup has a unique minimal representative, any such member of a coset of $W_J$ must also be the minimal element of the coset of $W_I$ containing it. 

For each $J \subseteq S$, we may equip $W^J$ with a total ordering. We will denote such an ordering by $\ll^J$ and define $L_\ll^J(w^J) = |\{u^J \in W^J \, | \, u^J \ll^J w^J\}|$. Analogous to $\varphi_\ll$, we may define $\varphi_\ll^J : W \hookrightarrow \Sym{|W^J|}$ by $$(L_\ll^J(u^J))\varphi_\ll^J(w) = L_\ll^J((u^Jw^{-1})^J)$$ for all $u^J \in W^J$ and $w \in W$. For all $w \in W^J$ such that $ws \in W^J$ also, we have $$\varphi_\ll^J(s) = \prod_{w <_R ws} (L_\ll^J(w), L_\ll^J(ws)).$$

These maps will be embeddings provided $W_J$ is a core-free subgroup of $W$.

When $J = \emptyset$,  the definitions of $\ll^\emptyset, L_\ll^\emptyset$ and $\varphi_\ll^\emptyset$ coincide with those of $\ll, L_\ll$ and $\varphi$ respectively. For convenience then, we omit the appearance of $J = \emptyset$ in their notation.


\section{Strong E-embeddings}\label{Section Main Results}

Our first aim of this section is to prove that strong E-embeddings are indeed E-embeddings. This is done in Proposition \ref{Proposition Strong E-embeddings are E-embeddings} with Lemma \ref{Lemma aux mutually commuting transpositions} acting as an auxiliary lemma.
\begin{lemma}\label{Lemma aux mutually commuting transpositions}
Let $t_1,\ldots,t_k \in \Sym{n}$ be mutually disjoint transpositions for some positive integer $n$. For all $w \in \Sym{n}$, if $w <_B wt_i$ for each $i = 1,\ldots,k$, then $w <_B w t_1,\ldots t_k$.
\end{lemma}
\begin{proof}
    Write $t_i = (a_i, b_i)$ for each $i = 1,\ldots,k$. Since $t_1,\ldots,t_k$ are disjoint, $a_i$ and $b_i$ are moved by only $t_i$. Equivalently, $\{a_i,b_i\} \cap \{a_j,b_j\} = \emptyset$ for all distinct pairs $i,j = 1,\ldots,k$. We may now iteratively apply Lemma \ref{Lemma Bruhat in Symmetric group} to achieve the chain $w <_B wt_1 <_B wt_1t_2 <_B \ldots <_B w t_1,\ldots t_k$.
\end{proof}

\begin{proposition}\label{Proposition Strong E-embeddings are E-embeddings} Suppose $W$ is a Coxeter group and  let $\varphi: W \hookrightarrow \Sym{n}$ be an embedding for some positive integer $n$. If $\varphi$ is a strong E-embedding, then $\varphi$ is an E-embedding.
\end{proposition}
\begin{proof}
    Suppose $\varphi$ is a strong E-embedding and consider $w \in W$ and $s \in S$ such that $w <_R ws$. If $\varphi(s) = t_1 \ldots t_k$ as the product of disjoint transpositions, then, since $\varphi$ is a strong E-embedding, $ \varphi(w) <_B \varphi(w)t_i$ for each $i=1,\ldots,k$. By Lemma \ref{Lemma aux mutually commuting transpositions}, $\varphi(w) <_B \varphi(w)t_1\ldots t_k = \varphi(ws)$.
\end{proof}

Now we prove Proposition \ref{Proposition Strong E-embeddings Weak Order Bruhat Equivalence} where we show that strong E-embeddings can be reformulated in terms of the Bruhat order only.

\begin{proof}[Proof of Proposition \ref{Proposition Strong E-embeddings Weak Order Bruhat Equivalence}]
    Since $<_B$ is a refinement of $<_R$, $(ii)$ necessarily implies $(i)$.

    Now suppose that $\varphi$ is a strong E-embedding and take $w \in W$, $t \in T$ such that $w \triangleleft_B wt$. Since $w \triangleleft_B wt$, Lemma \ref{Lemma Bruhat Subwords} informs us that there exists $u,v \in W$ and $s \in S$ such that $wt = usv^{-1}$, $w = uv^{-1}$. Some consequences of this are:
    \begin{itemize}
        \item $\ell(wt) = \ell(u)+\ell(s)+\ell(v^{-1})$;
        \item $u <_R us$ and $v <_R vs$; and 
        \item $t = s^v = vsv^{-1}$.
    \end{itemize}
    
We may write $\varphi(t)$ and $\varphi(s)$ as products of disjoint  transpositions thus $\varphi(t) = t_1 \dots t_k$ and $\varphi(s) = s_1 \dots s_k$, where $t_i = (a_i,b_i)$ and $s_i =(\alpha_i,\beta_i)$ for $i = 1, \dots, k$. Since $t = s^v = vsv^{-1}$, conjugation in $\Sym{n}$ ensures that both $\varphi(t)$ and $\varphi(s)$ are products of the same number of transpositions, $k$, and we may index in such a way that $(\alpha_i)\varphi(v^{-1}) = a_i$ and $(\beta_i)\varphi(v^{-1}) = b_i$.

    It remains to show that $\varphi(w) <_B \varphi(w)t_i$ for each $i=1,\ldots,k$. By Lemma \ref{Lemma Bruhat in Symmetric group},  $\varphi(w) <_B \varphi(w)t_i$ if and only if $(a_i)\varphi(w)^{-1} < (b_i)\varphi(w)^{-1}$. 
    Now the following are equivalent.
    \begin{align*}
    \varphi(w) &<_B \varphi(w)t_i\\
        (a_i)\varphi(w)^{-1} &< (b_i)\varphi(w)^{-1} \\
        (\alpha_i)\varphi(v^{-1})\varphi(w)^{-1} &< (\beta_i)\varphi(v^{-1})\varphi(w)^{-1} \\
        (\alpha_i)\varphi(v^{-1})\varphi(uv^{-1})^{-1} &< (\beta_i)\varphi(v^{-1})\varphi(uv^{-1})^{-1} \\
       (\alpha_i)\varphi(v^{-1}vu^{-1}) &< (\beta_i)\varphi(v^{-1}vu^{-1}) \\
        (\alpha_i)\varphi(u)^{-1} &< (\beta_i)\varphi(u)^{-1} \\
        \varphi(u) &<_B \varphi(u)s_i.\\
    \end{align*}
    Since $u <_R us$, $\varphi(u) <_B \varphi(u)s_i$ follows from the fact that $\varphi$ is a strong E-embedding, which completes the proof.
\end{proof}

Now we focus on proving Theorem \ref{Theorem Bruhat Equivelence} using the auxiliary lemma, Lemma \ref{Lemma Strong E-embedding  Restriction}.  

\begin{lemma}\label{Lemma Strong E-embedding  Restriction}
    Suppose $I \subseteq J \subseteq S$ and that $\ll^I$ and $\ll^J$ are total refinements of $<_B$ restricted to $W^I$ and $W^J$ respectively. Suppose further that $\ll^I$ and $\ll^J$ agree on $W^J$: for all $u,v \in W^J$, $u \ll^I v$ if and only if $u \ll^J v$. Then if $\varphi_\ll^I$ is a strong E-embedding, so is $\varphi_\ll^J$.
\end{lemma}
\begin{proof}
Recall that since $I \subseteq J$, it follows that $W^J \subseteq W^I$.
Now suppose that $\varphi_\ll^I$ is a strong E-embedding and take some $w \in W$ and $s \in S$ such that $w <_R ws$. Since $\varphi_\ll^I$ is an E-embedding, $$\varphi_\ll^I(w) <_B \varphi_\ll^I(w)(L_\ll^I(u), L_\ll^I(us))$$ for all $u,us \in W^I$ such that $u^I \ll^I (us)^I$. Equivalently, by Lemma \ref{Lemma Bruhat in Symmetric group}, $(uw^{-1})^I \ll^I (usw^{-1})^I$. But since $\ll^I$ and $\ll^J$ agree on $W^J \subseteq W^I$, it's true that $(uw^{-1})^J \ll^J (usw^{-1})^J$ also for all $u,us \in W^J$ such that $u^J \ll^J (us)^J$. Therefore $\varphi_\ll^I$ is also a strong E-embedding.
\end{proof}

\begin{proof}[Proof of Theorem \ref{Theorem Bruhat Equivelence}.]
    First, we show that total refinements of the Bruhat order give rise to strong E-embeddings. We focus on the case when $J = \emptyset$ before extending it to all $J \subseteq S$. 
    
    For all $w \in W$ and $t \in T$ such that $w <_B wt$, we need to show that $$\varphi(w) <_B \varphi(w) (L_\ll(u), L_\ll(ut))$$ for all $u \in W$ such that $u <_B ut$. Since $\ll$ is a total refinement of $<_B$, $L_\ll(u) < L_\ll(ut)$ and thus proving $\varphi(w) < \varphi(w) (L_\ll(u), L_\ll(ut))$ is equivalent to showing $uw^{-1} \ll utw^{-1}$. But by Lemma \ref{Lemma Bruhat t sandwich}, $uw^{-1} <_B utw^{-1}$ and so $uw^{-1} \ll utw^{-1}$. 

     Now we extend the results to all $J \subseteq S$. Every total refinement of the Bruhat order on $W^J$ is a restriction of some total refinement of the Bruhat order on $W$. The implication now follows from Lemma \ref{Lemma Strong E-embedding  Restriction}.
    
    Conversely, suppose for contradiction, that $\varphi_\ll^J$ is a strong E-embedding but $\ll^J$ is \textit{not} a refinement of the Bruhat order restricted to $W^J$. Then there exists some $w\in W^J$, $t \in T$ such that $wt \in W^J$, $w <_B wt$ but $wt \ll^J w$. Since $\varphi_\ll^J$ is a strong E-embedding, $w <_B wt$ implies that $$\varphi_\ll^J(w) <_B \varphi_\ll^J(w)(L_\ll^J(wt), L_\ll^J(w)).$$
    But by Lemma \ref{Lemma Bruhat in Symmetric group}, this is equivalent to $(wtw^{-1})^J \ll(ww^{-1})^J$. But this contradicts the fact that $L_\ll^J(1)=0$.
\end{proof}

Theorem \ref{Theorem Bruhat Refining iff see} follows from Theorem \ref{Theorem Bruhat Equivelence} on setting $J = \emptyset$.

\section{Constructing Elnitsky's Tilings}\label{Section Constructing Elnitsky's Tilings}
Using Theorem \ref{Theorem Bruhat Equivelence} we now have a construction for producing new E-tilings for any reduced word from any finite Coxeter group, thus proving Theorem \ref{maintheorem}
\begin{proof}
    The construction works as follows:
\begin{enumerate}[$(i)$]
    \item Choose some parabolic subgroup $W_J < W$ of our Coxeter group $W$ which is core-free (and so $J \ne S$);
    \item Produce a total order $\ll$ of  $W^J$ that is a refinement of the Bruhat order. A particular method of doing this is outlined in \cite{nicolaidesrowley3};
    \item Construct $\phi_\ll^J$ as defined in Section 2. By Theorem \ref{Theorem Bruhat Equivelence} this is a strong E-embedding;
    \item By Proposition \ref{Proposition Strong E-embeddings are E-embeddings}, $\phi_\ll^J$ is an E-embedding, and therefore by Theorem 1.2 of \cite{nicolaidesrowley1}, for any reduced word of any element of $W$, there exists an associated E-tiling.
\end{enumerate}
\end{proof}

We stress that the proof of Theorem \ref{maintheorem} does not suggest that all E-embeddings must be constructed in this way, although we are yet to find a counterexample.

In the remainder of this section we apply the proof of Theorem \ref{maintheorem} in three scenarios. First, we re-examine Elnitsky's original tilings, realising them as strong-E-Embeddings. Next, we concentrate on examples of reducible finite Coxeter groups. Finally, we give two new tilings for the longest element of $\mathrm{E}_8$.


\subsection{Tilings for $\mathrm{D}_5$}
In \cite{nicolaidesrowley1}, it is shown how each E-embedding has an associated tiling system. Coincidentally, all of Elnitsky's original tilings correspond to strong-E-embeddings that are actions on the distinguished elements of specific parabolic subgroups. We showcase this for the group $\mathrm{D}_5$ but note that the details for realizing his remaining tiling systems as strong-E-embeddings are much the same. 

Let $W$ be the Coxeter group $\mathrm{D}_5$ with $\Gamma(W)$ given by  \newline
\begin{center}
    
\begin{tikzpicture}[scale=0.5]
\draw (2,0-4-4) -- (6,0-4-4);
\draw (2,0-4-4) -- (0,1-4-4);
\draw (2,0-4-4) -- (0,-1-4-4);

\node (s1) at (0,1-4-4-0.6) {$s_1$};
\node (s1) at (0,-1-4-4-0.6) {$s_2$};
\node (s1) at (2,-0.6-4-4) {$s_3$};
\node (s1) at (4,-0.6-4-4) {$s_4$};
\node (s1) at (6,-0.6-4-4) {$s_5$};

\filldraw [black] (0,1-4-4) circle (5pt);
\filldraw [black] (0,-1-4-4) circle (5pt);
\filldraw [black] (2,-0-4-4) circle (5pt);
\filldraw [black] (4,-0-4-4) circle (5pt);
\filldraw [black] (6,-0-4-4) circle (5pt);
\end{tikzpicture}
\end{center}

We give a very brief overview (with minor adjustments for convenience) of Elnitsky's tiling system for $\mathrm{D}_5$, valuing brevity over depth. For a more thorough and comprehensive introduction, we point the reader to \cite{elnitskyThesis}, \cite{elnitsky},  \cite{nicolaidesrowley2}, and  \cite{tenner1}. 

Consider the embedding of $\mathrm{D}_5$ into $\Sym{10}$ where $s_1 = (5,7)(4,6)$ and $s_i = (6 - i, 7 - i )(4 + i, 5 + i)$ for each $i \in \{2, \dots, 5 \}$. As a permutation representation, it's represented by the graph in Figure \ref{Figure D5 sym10}.
\begin{figure}[H]
   \centering
\begin{tikzpicture}[scale=0.9]
\draw (2,0-4-4) -- (8,0-4-4);
\draw (2,0-4-4) -- (0,1-4-4);
\draw (2,0-4-4) -- (0,-1-4-4);
\draw (-2,0-4-4) -- (-8,0-4-4);
\draw (-2,0-4-4) -- (0,1-4-4);
\draw (-2,0-4-4) -- (0,-1-4-4);

\filldraw [white, draw=black] (0,1-4-4) circle (8pt);
\filldraw [white, draw=black] (0,-1-4-4) circle (8pt);
\filldraw [white, draw=black] (2,-0-4-4) circle (8pt);
\filldraw [white, draw=black] (4,-0-4-4) circle (8pt);
\filldraw [white, draw=black] (6,-0-4-4) circle (8pt);
\filldraw [white, draw=black] (8,-0-4-4) circle (8pt);
\filldraw [white, draw=black] (-2,-0-4-4) circle (8pt);
\filldraw [white, draw=black] (-4,-0-4-4) circle (8pt);
\filldraw [white, draw=black] (-6,-0-4-4) circle (8pt);
\filldraw [white, draw=black] (-8,-0-4-4) circle (8pt);

\node (s1) at (0,1-4-4-0.0) {$5$};
\node (s1) at (0,-1-4-4-0.0) {$6$};
\node (s1) at (2,-0.0-4-4) {$7$};
\node (s1) at (4,-0.0-4-4) {$8$};
\node (s1) at (6,-0.0-4-4) {$9$};
\node (s1) at (8,-0.0-4-4) {$10$};
\node (s1) at (-2,-0.0-4-4) {$4$};
\node (s1) at (-4,-0.0-4-4) {$3$};
\node (s1) at (-6,-0.0-4-4) {$2$};
\node (s1) at (-8,-0.0-4-4) {$1$};

\node (s1) at (0+1.25,1-4-4-.25) {$s_1$};
\node (s1) at (0-1.25,-1-4-4-0.0+0.25) {$s_1$};
\node (s1) at (0-1.25,1-4-4-.25) {$s_2$};
\node (s1) at (0+1.25,-1-4-4-0.0+0.25) {$s_2$};
\node (s1) at (3,-0.0-4-4 - 0.3) {$s_3$};
\node (s1) at (-3,-0.0-4-4 - 0.3) {$s_3$};
\node (s1) at (5,-0.0-4-4 - 0.3) {$s_4$};
\node (s1) at (-5,-0.0-4-4 - 0.3) {$s_4$};
\node (s1) at (7,-0.0-4-4 - 0.3) {$s_5$};
\node (s1) at (-7,-0.0-4-4 - 0.3) {$s_5$};

\end{tikzpicture} 
   \caption{Permutation representation of $\mathrm{D_5}$ in $\Sym{10}$.}
   \label{Figure D5 sym10}
\end{figure}

For all $w \in \mathrm{D}_5$, define $\mathrm{P}(w)$ to be the (possibly deformed) 20-gon made like so:
\begin{enumerate}[$(i)$]
    \item Let $\mathrm{L}$ be the lower-most vertex of our $20$-gon. 
    \item Let the first $10$ edges clockwise from $\mathrm{L}$ be those of the regular $20$-gon with unit length edges, labelling them with $\{1,2,\ldots,10\}$ respectively. Let $\mathrm{X}$ denote the horizontal line passing through the $5^{\text{th}}$ vertex clockwise from $\mathrm{L}$ in this way.
    \item  Construct and label the $10$ edges anti-clockwise from $\mathrm{L}$  so that the $i^{th}$ edge from $\mathrm{L}$ in this direction is a unit length edge parallel to the edge labelled $(i)w^{-1}$ from step $(ii)$. 
\end{enumerate}

We present an example in Figure \ref{PolygonD10} for $w = (1,10)(2,9)(3,8)(4,7)$.

\begin{figure}
    \centering
    \includegraphics[height = 9cm]{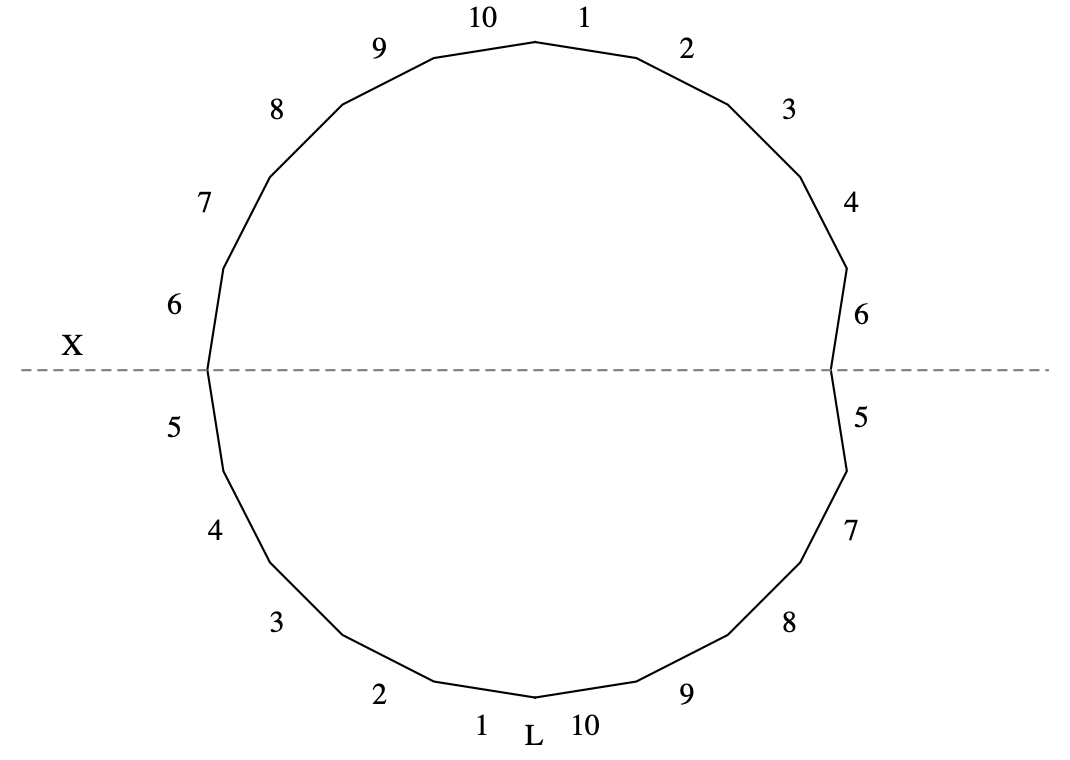}
\caption{The polygon $\mathrm{P}((1,10)(2,9)(3,8)(4,7))$.}
    \label{PolygonD10}
\end{figure}

For each $w \in W$, Elnitsky is interested in tilings of $\mathrm{P}(w)$ that are made by a restricted set of tiles satisfying certain rules. The tiles placed can either be rhombi or a so-called \textit{megatile} (described in detail below) but the tiling must be symmetric about $\mathrm{X}$.

A megatile is an octagonal tile satisfying the following properties: 
\begin{enumerate}[$(i)$]
\item Its edges are unit length and it must be placed so that it is symmetric through $\mathrm{X}$;
    \item Its upper-most vertex, $\mathrm{U}_0$, and lower-most vertex, $\mathrm{L}_0$, must lie on a vertical line;
    \item Let $E_0$ denote the first four edges anti-clockwise from $\mathrm{U}_0$, then the remaining four edges can be obtained like so: reflect $E_0$ through the vertical line passing through $\mathrm{U}_0$ and $\mathrm{L}_0$ before 
    transposing the resulting first and second pair of edges, and the third and fourth pair of edges.
\end{enumerate}

We present two such tilings for the same polygon in Figure \ref{TilingsD10}.

\begin{figure}
\begin{subfigure}[b]{0.5\textwidth}
    \centering
    \includegraphics[height = 6cm]{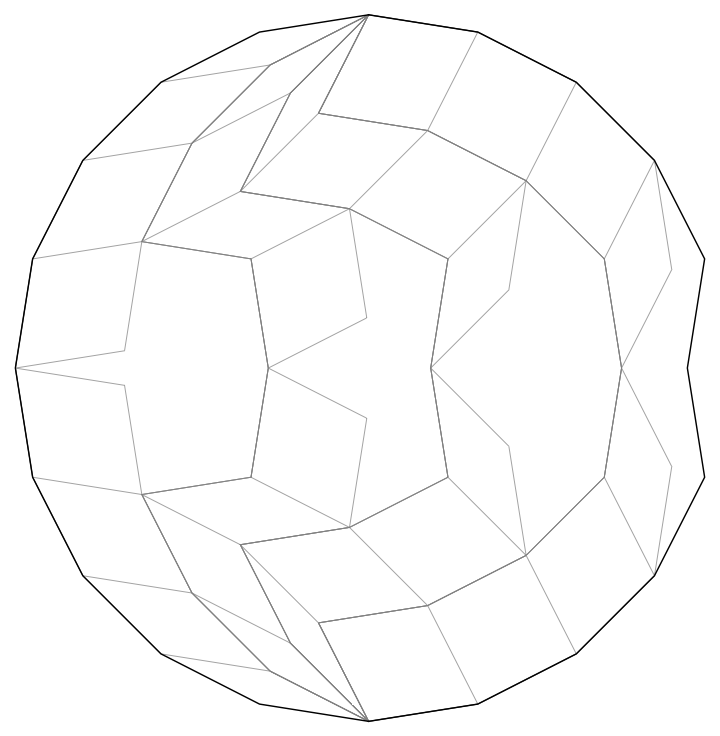}
\end{subfigure}
\begin{subfigure}[b]{0.5\textwidth}
    \centering
    \includegraphics[height = 6cm]{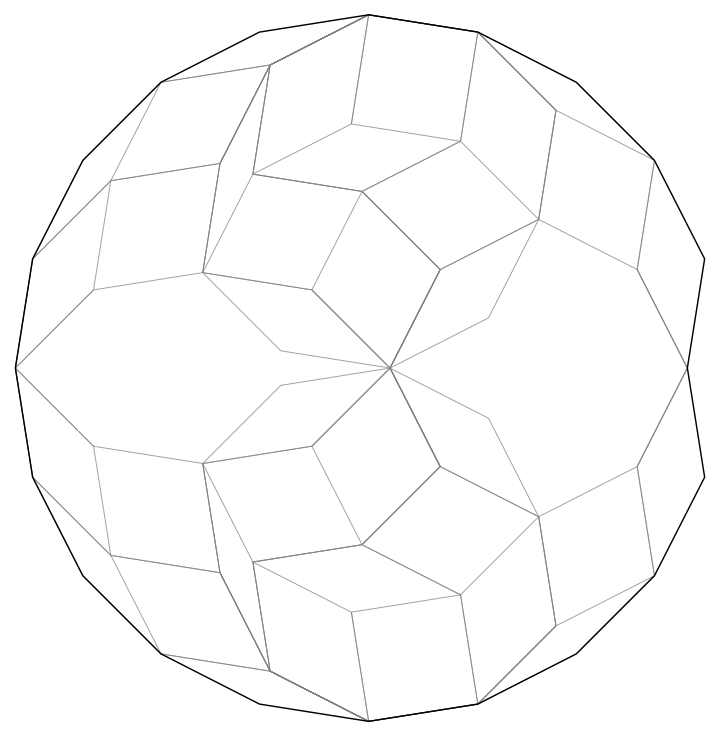}
\end{subfigure}
\caption{Two Elnitsky tilings for $\mathrm{P}((1,10)(2,9)(3,8)(4,7))$.}
    \label{TilingsD10}
\end{figure}

Theorem 7.1 of \cite{elnitsky} gives a bijection between classes of the reduced words of $\mathrm{D}_n$ and tilings of the kind described above. To produce a tiling from a reduced word, the word is read from left to right with each subsequent generator corresponding to instructions for placing the next set of tiles. Reading the generator $s_1$ compels us to place a megatile with the edges $E_0$ of the megatile are formed by taking the right-most edges that are currently the $4^{th}, 5^{th}, 6^{th}$ and $7^{th}$ edges from $\mathrm{L}$ respectively. For $i>1$, $s_i$ is associated with placing a rhombic tile sharing the right-most edges placed in the $i^{th}$ and $(i-1)^{th}$ position from $\mathrm{L}$ and its mirror image through $\mathrm{X}$. The two Elnitsky tiling s of Figure \ref{Figure D5 sym10} are derived from the reduced words  $$s_5 s_4 s_3 s_2 s_1 s_5 s_4 s_3 s_2 s_1 s_5 s_4 s_3 s_2 s_1 s_5 s_4 s_3 s_2 s_1$$ and $$s_5 s_3 s_4 s_2 s_3 s_1 s_2 s_4 s_3 s_2 s_5 s_4 s_5 s_3 s_4 s_2 s_1 s_3 s_4 s_2,$$ respectively. We give an explicit step-by-step illustration of how to produce the tiling for $s_5 s_3 s_4 s_2 s_3 s_1 s_2 s_4 s_3 s_2 s_5 s_4 s_5 s_3 s_4 s_2 s_1 s_3 s_4 s_2$ in Figure \ref{Figure TilingD10Sequence}.

\begin{figure}
    \centering
    \includegraphics[width = 15cm]{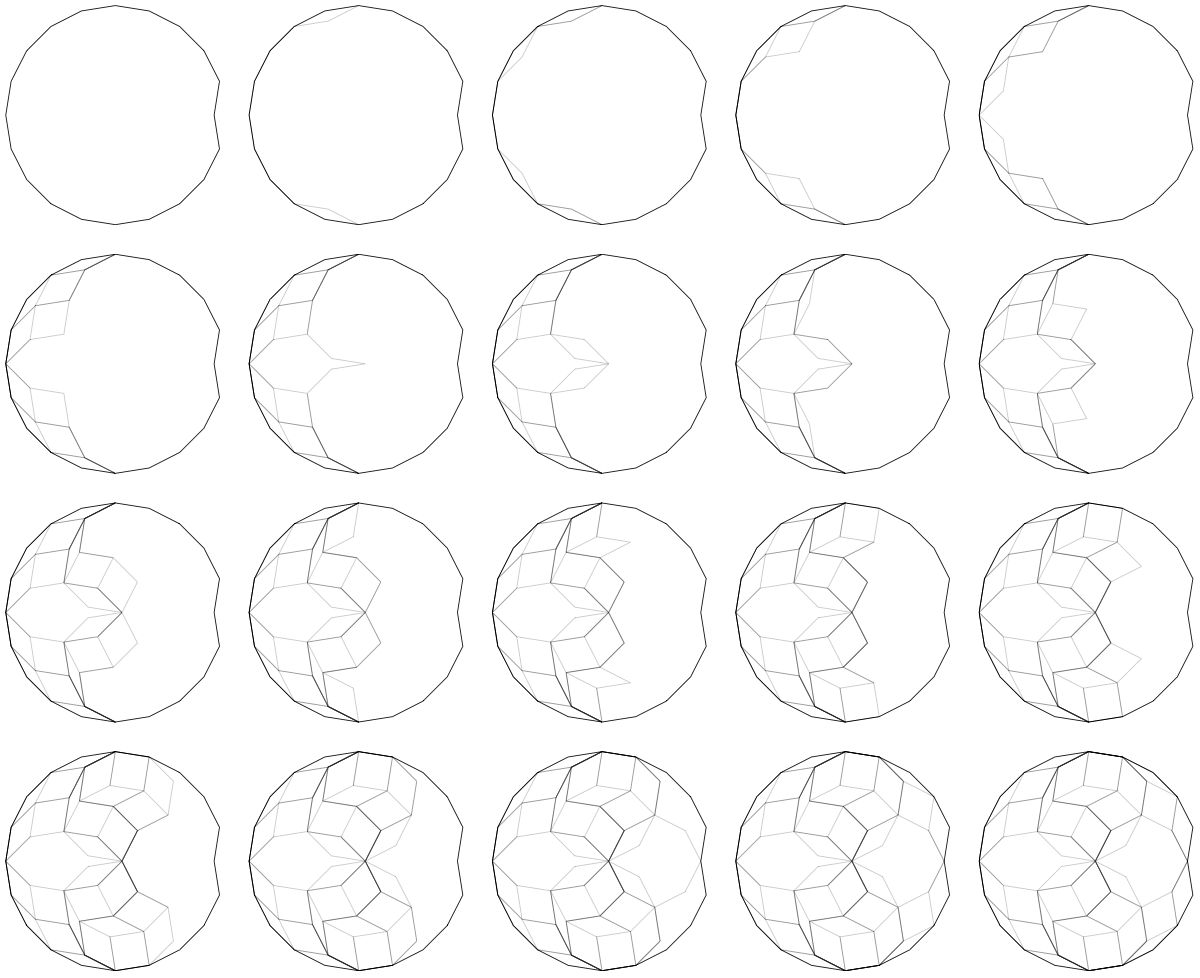}
\caption{A step-by-step illustration of the intermediate tilings associated to the reduced word $s_5 s_3 s_4 s_2 s_3 s_1 s_2 s_4 s_3 s_2 s_5 s_4 s_5 s_3 s_4 s_2 s_1 s_3 s_4 s_2$ using Elnitsky's method (read from left to right, top to bottom).}
    \label{Figure TilingD10Sequence}
\end{figure}

We now pivot to reconstructing the same tilings from the perspective of strong-E-embeddings and the Bruhat order. Start by considering $\mathrm{D}_4$ as the index 10, standard parabolic subgroup of $\mathrm{D}_5$. That is, let $J = \{s_1,s_2,s_3,s_4\}$ so  $W_J= \langle s_1,s_2,s_3,s_4\rangle$. We display the Hasse Diagram of the Bruhat order restricted to $W^J$ in Figure \ref{Figure D4 D5 Bruhat}: one element lies vertically above another if and only if the higher element is greater than the lower in the Bruhat order (whilst elements of the same height are incomparable in the Bruhat Order). 

\begin{figure}[H]
   \centering
\begin{tikzpicture}[scale=1.5]
\draw (0,0) -- (0,3);
\draw (0,3) -- (0.7, 3.7);
\draw (0,3) -- (-0.7, 3.7);
\draw (0.7, 3.7) -- (0,4.4);
\draw (-0.7, 3.7) -- (0,4.4);
\draw (0,4.4) -- (0,7.4);

\filldraw [black, draw=black] (0,0) circle (2pt);
\filldraw [black, draw=black] (0,1) circle (2pt);
\filldraw [black, draw=black] (0,2) circle (2pt);
\filldraw [black, draw=black] (0,3) circle (2pt);
\filldraw [black, draw=black] (0.7, 3.7) circle (2pt);
\filldraw [black, draw=black] (-0.7, 3.7) circle (2pt);
\filldraw [black, draw=black] (0,4.4) circle (2pt);
\filldraw [black, draw=black] (0,5.4) circle (2pt);
\filldraw [black, draw=black] (0,6.4) circle (2pt);
\filldraw [black, draw=black] (0,7.4) circle (2pt);

\node (s1) at (0.5,-0.25) {1};
\node (s1) at (0.55,1-0.25) {$s_5$};
\node (s1) at (0.6,2-0.25) {$s_5s_4$};
\node (s1) at (0.65,3-0.25) {$s_5s_4s_3$};
\node (s1) at (1.5,3.7-0.25) {$s_5s_4s_3s_2$};
\node (s1) at (-1.5,3.7-0.25) {$s_5s_4s_3s_1$};
\node (s1) at (0.9,5-0.4) {$s_5s_4s_3s_1s_2$};
\node (s1) at (1,6-0.4) {$s_5s_4s_3s_1s_2s_3$};
\node (s1) at (1.1,7-0.4) {$s_5s_4s_3s_1s_2s_3s_4$};
\node (s1) at (1.2,8-0.4) {$s_5s_4s_3s_1s_2s_3s_4s_5$};

\end{tikzpicture} 
   \caption{The Hasse diagram of the Bruhat Order restricted to the minimal right cosets representatives of $\mathrm{D}_4$ in $\mathrm{D}_5$.}
   \label{Figure D4 D5 Bruhat}
\end{figure}
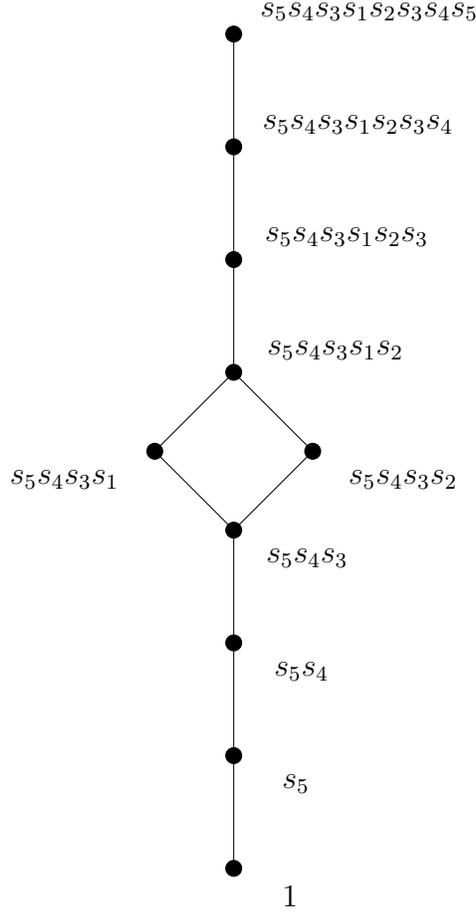

Now we need to choose some total ordering of our minimal representatives that is a refinement of the Bruhat order. There are exactly two such refinements:
\begin{enumerate}[$(i)$]
    \item $1 \ll s_5 \ll s_5s_4 \ll s_5s_4s_3 \ll s_5s_4s_3s_2 \ll s_5s_4s_3s_1 \ll s_5s_4s_3s_1s_2$
    $\ll s_5s_4s_3s_1s_2s_3 \ll s_5s_4s_3s_1s_2s_3s_4 \ll s_5s_4s_3s_1s_2s_3s_4s_5 $; 
    and
    \item $1 \ll s_5 \ll s_5s_4 \ll s_5s_4s_3 \ll s_5s_4s_3s_1 \ll s_5s_4s_3s_2 \ll s_5s_4s_3s_1s_2 $ $\ll s_5s_4s_3s_1s_2s_3 \ll s_5s_4s_3s_1s_2s_3s_4 \ll s_5s_4s_3s_1s_2s_3s_4s_5$.
\end{enumerate}

For this small example, finding and choosing a refinement is a manageable task. For larger examples, however, this proves to be a more difficult task to do satisfactorily. Fortunately, we can always use an instance of Prim's greedy algorithm to create such a refinement, namely the deletion order. Details of this are given in \cite{nicolaidesrowley3}.
Choosing the first of the two refinements to be $\ll$, we can make  $\varphi_\ll^J$ as in Section \ref{Section background}. Explicitly,
\begin{align*}
    \varphi_\ll^J(s_1) &= (L_\ll^J(s_5s_4s_3), L_\ll^J(s_5s_4s_3s_1)) (L_\ll^J(s_5s_4s_3s_2), L_\ll^J(s_5s_4s_3s_1s_2))\\
    &= (4,6)(5,7)\\
    \varphi_\ll^J(s_2) &= (L_\ll^J(s_5s_4s_3), L_\ll^J(s_5s_4s_3s_2)) (L_\ll^J(s_5s_4s_3s_1), L_\ll^J(s_5s_4s_3s_1s_2))\\
    &= (4,5)(6,7);\\
    \varphi_\ll^J(s_3) &= (L_\ll^J(s_5s_4), L_\ll^J(s_5s_4s_3)) (L_\ll^J(s_5s_4s_3s_1s_2), L_\ll^J(s_5s_4s_3s_1s_2s_3))\\
    &= (3,4)(7,8);\\
    \varphi_\ll^J(s_4) &= (L_\ll^J(s_5), L_\ll^J(s_5s_4)) (L_\ll^J(s_5s_4s_3s_1s_2s_3), L_\ll^J(s_5s_4s_3s_1s_2s_3s_4))\\
    &= (2,3)(8,9);\\
    \varphi_\ll^J(s_4) &= (L_\ll^J(1), L_\ll^J(s_5)) (L_\ll^J(s_5s_4s_3s_1s_2s_3s_4), L_\ll^J(s_5s_4s_3s_1s_2s_3s_4s_5))\\
    &= (1,2)(9,10).
\end{align*}  The permutation representation graph of $\varphi_\ll^J$ is therefore the same as that used in Elnitsky's construction in Figure \ref{Figure D5 sym10}. By Theorem \ref{Theorem Bruhat Refining iff see}, $\varphi_\ll^J$ is a strong-E-embedding and thus, by Proposition \ref{Proposition Strong E-embeddings are E-embeddings}, an E-embedding. Using the methods in \cite{nicolaidesrowley1}, we can derive the tiling system associated to $\varphi_\ll^J$ that transforms reduced words into E-tilings. This system is exactly the same as Elnitsky's tiling method for type $\mathrm{D}$.

Elnitsky's other tiling systems can be realized using the standard parabolic subgroups $\mathrm{A}_{n-1}$ in $\mathrm{A}_n$ and $\mathrm{A}_{n-1}$ in $\mathrm{B}_n$.
\subsection{Tilings of finite reducible Coxeter groups}
Here we show how the steps outlined in the proof of Theorem \ref{maintheorem} work equally well for the finite reducible Coxeter groups as opposed to their irreducible counterparts. By way of example, we consider the reducible Coxeter group $W = \mathrm{A}_2 \times \mathrm{A}_3$. We label and order the generators like so:

\begin{center}
    
\begin{tikzpicture}[scale=0.5]
\draw (-2,0-4-4) -- (0,0-4-4);

\draw (2,0-4-4) -- (6,0-4-4);

\node (s1) at (-2,-0.6-4-4) {$s_1$};
\node (s1) at (0,-0.6-4-4) {$s_2$};

\node (s1) at (2,-0.6-4-4) {$s_3$};
\node (s1) at (4,-0.6-4-4) {$s_4$};
\node (s1) at (6,-0.6-4-4) {$s_5$};

\filldraw [black] (-2,-0.0-4-4) circle (5pt);
\filldraw [black] (0,-0.0-4-4) circle (5pt);
\filldraw [black] (2,-0-4-4) circle (5pt);
\filldraw [black] (4,-0-4-4) circle (5pt);
\filldraw [black] (6,-0-4-4) circle (5pt);
\end{tikzpicture}
\end{center}

Taking $J = \{s_1,s_3,s_4\}$, the Bruhat order restricted to $W^J$ is depicted in Figure \ref{Figure Bruhat A2A3}.

\begin{figure}[H]
    \centering
    
\begin{tikzpicture}[scale=1, rotate = 45]

\filldraw [black] (0,0) circle (5pt);
\filldraw [black] (0,4) circle (5pt);
\filldraw [black] (0,8) circle (5pt);
\filldraw [black] (0,12) circle (5pt);
\filldraw [black] (4,0) circle (5pt);
\filldraw [black] (4,4) circle (5pt);
\filldraw [black] (4,8) circle (5pt);
\filldraw [black] (4,12) circle (5pt);
\filldraw [black] (8,0) circle (5pt);
\filldraw [black] (8,4) circle (5pt);
\filldraw [black] (8,8) circle (5pt);
\filldraw [black] (8,12) circle (5pt);

\draw (0,0) -- (0,12);
\draw (4,0) -- (4,12);
\draw (8,0) -- (8,12);
\draw (0,0) -- (8,0);
\draw (0,4) -- (8,4);
\draw (0,8) -- (8,8);
\draw (0,12) -- (8,12);

\node (s1) at (0-0.6,0-0.8) {$1$};
\node (s1) at (4-0.6,0-0.8) {$s_2$};
\node (s1) at (8+0.6,0+0.8) {$s_2s_1$};
\node (s1) at (0-0.8,4-0.8) {$s_5$};
\node (s1) at (4-0.8,4-0.8) {$s_2s_5$};
\node (s1) at (8+0.8,4+0.8) {$s_2s_1s_5$};
\node (s1) at (0-1,8-0.8) {$s_5s_4$};
\node (s1) at (4-1,8-0.8) {$s_2s_5s_4$};
\node (s1) at (8+0.8,8+0.8) {$s_2s_1s_5s_4$};
\node (s1) at (0-1.2,12-0.8) {$s_5s_4s_3$};
\node (s1) at (4-1.2,12-0.8) {$s_2s_5s_4s_3$};
\node (s1) at (8+1.2,12+0.9 ){$s_2s_1s_5s_4s_3$};

\end{tikzpicture}

    \caption{A Hasse Diagram for the Bruhat order restricted to $W^J$ for $\mathrm{A}_2 \times \mathrm{A}_3$.}
    \label{Figure Bruhat A2A3}
\end{figure}
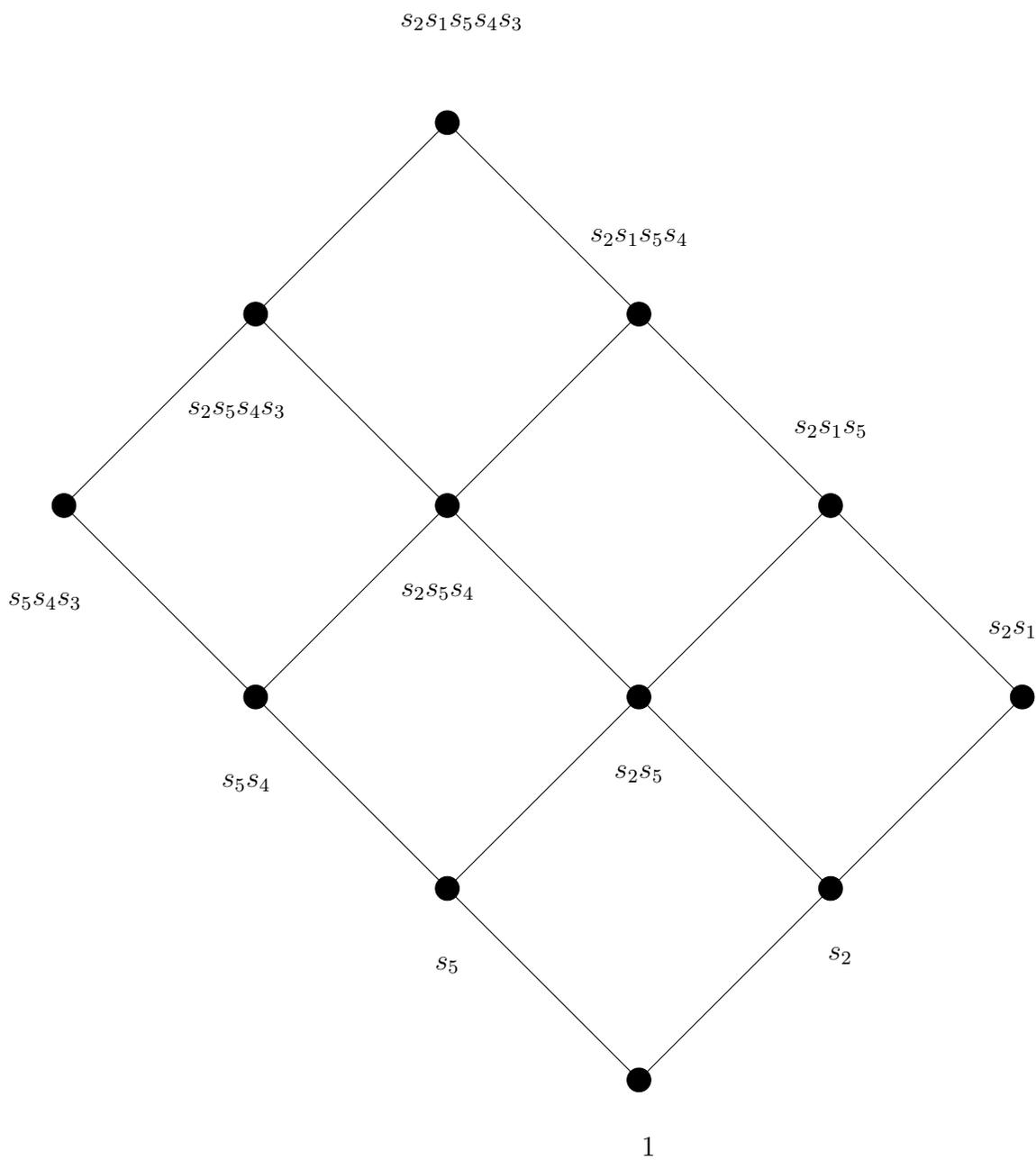

Taking the refinement $$1 \ll s_2 \ll s_2s_1 \ll s_5 \ll s_2s_5 \ll s_2s_1s_5 \ll s_5s_4 \ll s_2s_5s_4 $$ $$\ll s_2s_1s_5s_4 \ll s_5s_4s_3 \ll s_5s_4s_3s_2 \ll s_5s_4s_3s_2s_1,$$ say, we create the strong E-embedding $\varphi^J_\ll$ such that 
\begin{align*}
    \varphi^J_\ll(s_1) &= (2,3)(5,6)(8,9)(11,12);\\
    \varphi^J_\ll(s_2) &= (1,2)(4,5)(7,8)(10,11);\\
    \varphi^J_\ll(s_3) &= (7,10)(8,11)(9,12);\\
    \varphi^J_\ll(s_4) &= (4,7)(5,8)(6,9);\\
    \varphi^J_\ll(s_5) &= (1,4)(2,5)(3,6).
\end{align*}

Now we can use $\varphi^J_\ll$ to make E-tilings as Figure \ref{Figure A2A3 tiling} shows.

\begin{figure}[H]
    \centering
    \includegraphics[width = 8cm]{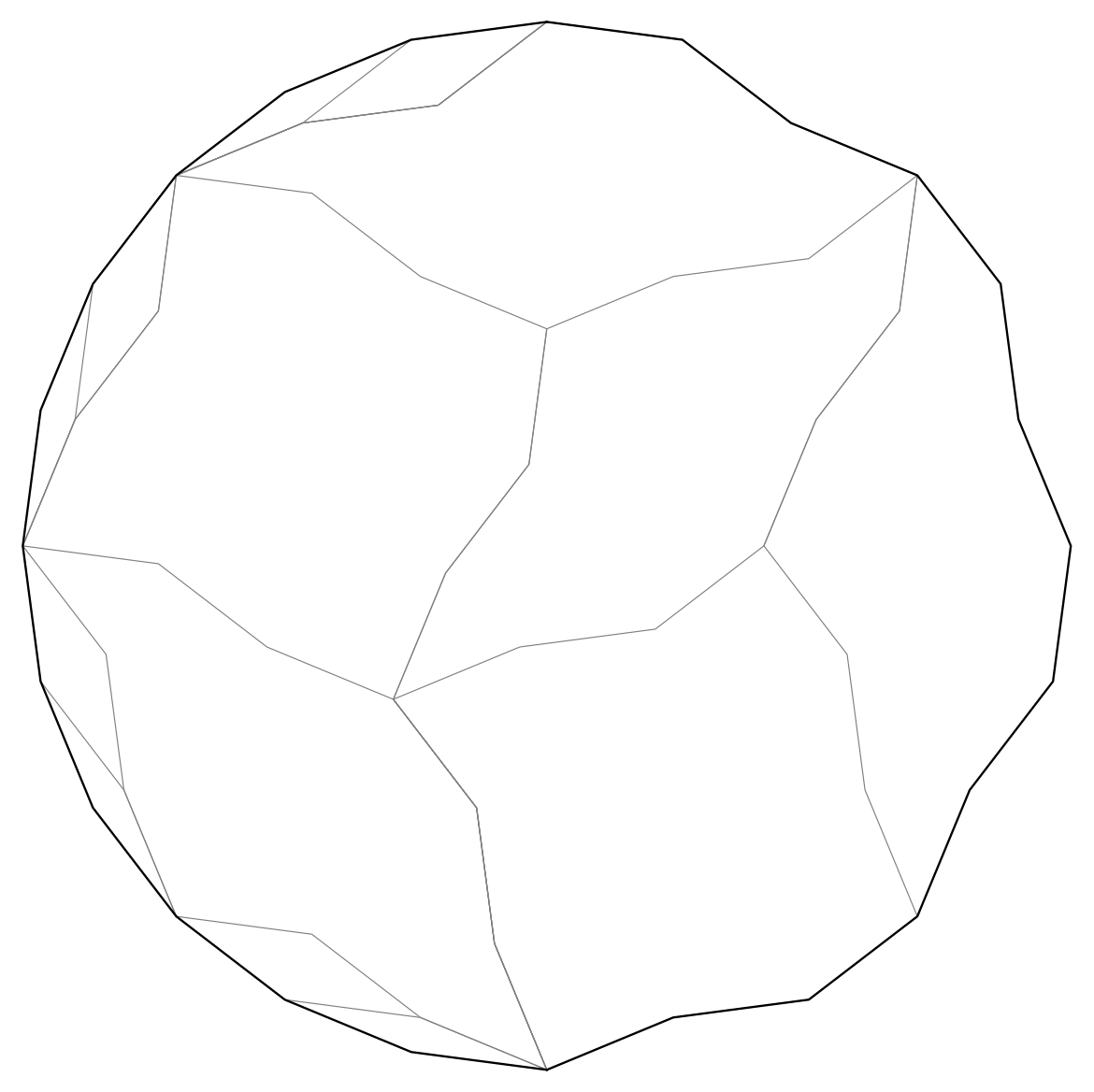}
    \caption{An E-tiling for the reduced word $s_1s_2s_3s_4s_5s_4s_3s_4$ in the Coxeter group $\mathrm{A}_2 \times \mathrm{A}_3$.}
    \label{Figure A2A3 tiling}
\end{figure}

\subsection{New tilings for $\mathrm{E}_8$}
Let $W$ be the Coxeter group $\mathrm{E}_8$ and let $\omega_0$ be the longest element of $W$.   We label $\Gamma(W)$ as  on page 16 of \cite{Iwohari-Hecke} and we order the elements of $S$ thus $s_1<s_2<\ldots<s_8$.  We follow our general procedure to produce new tilings  for $\mathrm{E}_8$ using the standard parabolic subgroup $\mathrm{E}_7$. 

We select two reduced words of $\omega_0$ to make our specific tilings.  Here $w_0$ denotes the longest element of $\mathrm{E}_8$. The first is made by taking the concatenation of 15 copies of $c=s_1s_2s_3s_4s_5s_6s_7s_8$ (a Coxeter element of $W$). One can prove that $\omega_0 = c^{15}$ and that this is indeed a reduced word. The tiling for this is given in Figure \ref{Figure E8 tiling cox elt}. The second tiling is produced by taking the least reduced word of $\omega_0$ with respect to the lexicographic (from right to left) ordering induced from the ordering of $S$. 

\begin{figure}
\begin{center}
\includegraphics[scale=0.7]{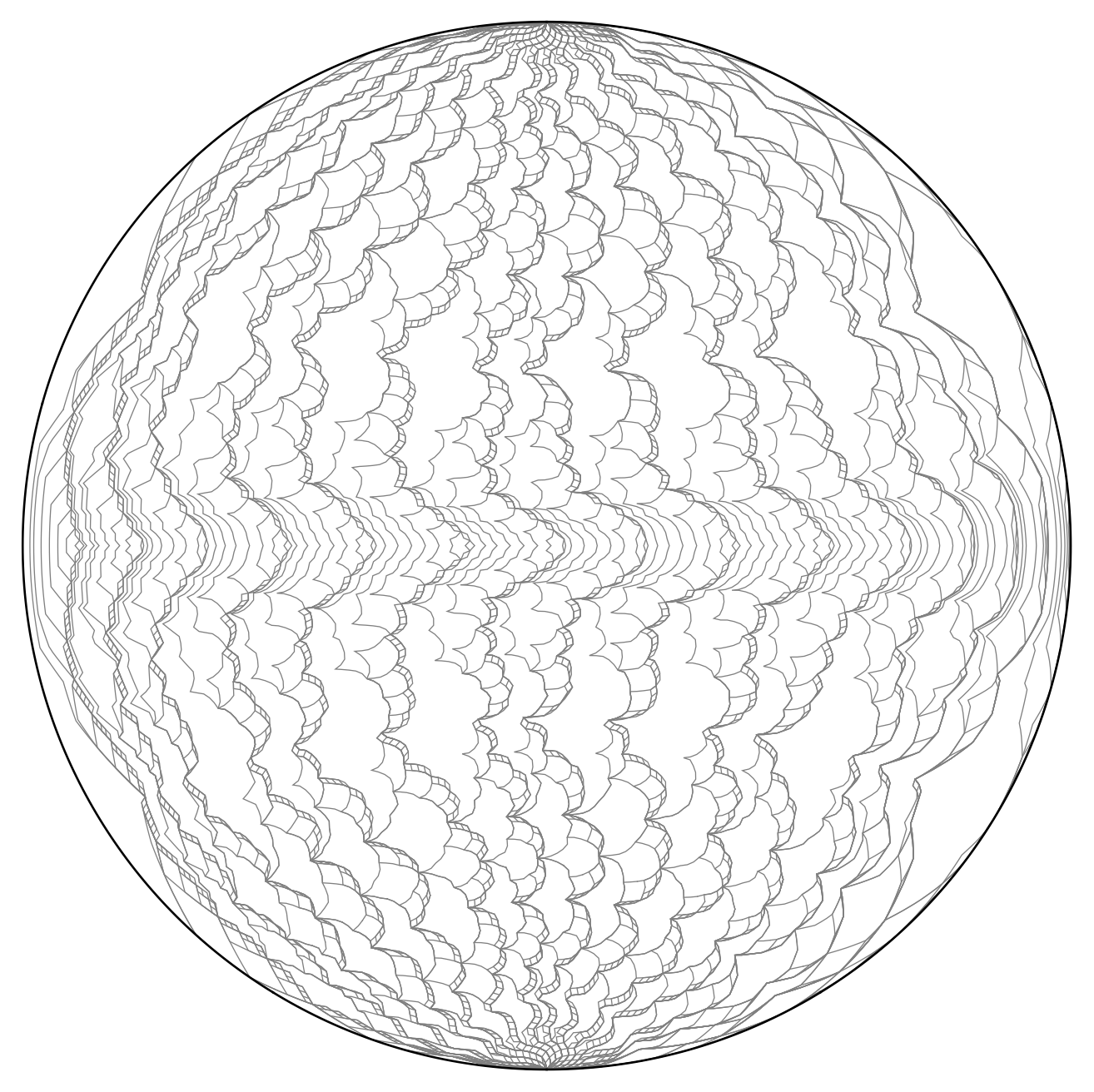}
\end{center}
\caption{E-tiling for the the reduced exprssion $c^{15}$ of $w_0$.}
\label{Figure E8 tiling cox elt}
\end{figure}
\newpage
\begin{figure}[H]
\begin{center}
\includegraphics[scale=0.7]{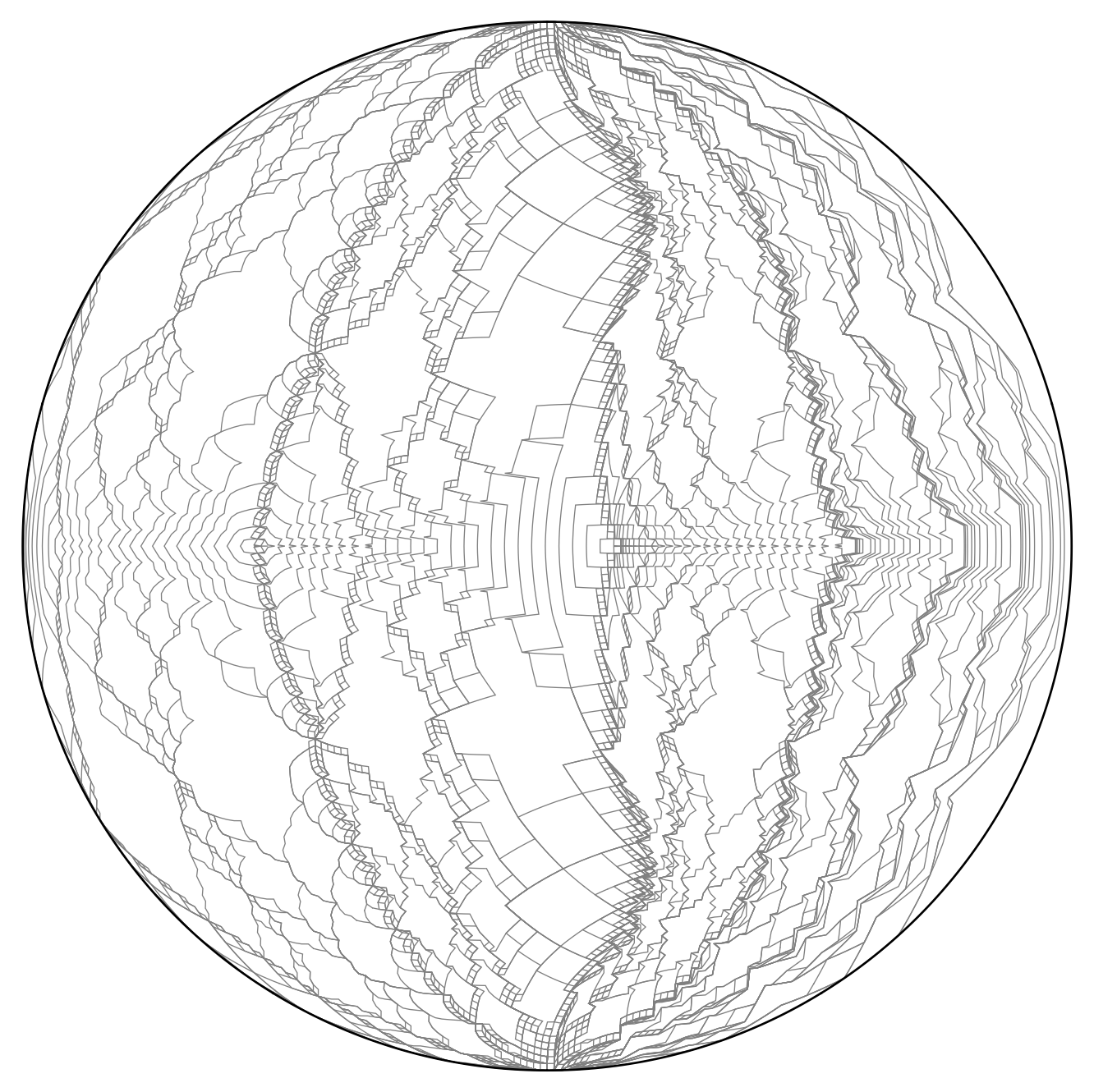}
\end{center}
\caption{E-tiling for the reduced expression least reduced word of $\omega_0$ with respect to the right to left lexicographic ordering.}

\end{figure}


\begin{thebibliography}{99}
\bibitem{cocg} Björner, Anders; Brenti, Francesco. Combinatorics of Coxeter groups. Graduate Texts in Mathematics, 231. Springer, New York, 2005. 

\bibitem{elnitskyThesis}  Elnitsky, Serge. Rhombic tilings of polygons and classes of reduced words in Coxeter groups. PhD dissertation, University of Michigan, 1993.
\bibitem{elnitsky}  Elnitsky, Serge. Rhombic tilings of polygons and classes of reduced words in Coxeter groups. J. Combin. Theory Ser. A 77 (1997), no. 2, 193–221.
\bibitem{epty} Escobar, Laura; Pechenik, Oliver; Tenner, Bridget Eileen; Yong, Alexander. Rhombic tilings and Bott-Samelson varieties. Proc. Amer. Math. Soc. 146 (2018), no. 5, 1921–1935. 
\bibitem{hamanaka} Hamanaka, H.; Nakamoto, A.; Suzuki, Y.. Rhombus Tilings of an Even-Sided Polygon and Quadrangulations on the Projective Plane. Graphs and Combinatorics 36 (2020), 561–571. 
\bibitem{Iwohari-Hecke} Geck, Meinolf; Pfeiffer, G\"otz.  Characters of Finite Coxeter Groups and Iwahori-Hecke Algebras. London Mathematical Society Monographs New Series,  21, Clarendon Press,  Oxford,  2000.
\bibitem{nicolaidesrowley1} Nicolaides, Robert; Rowley, Peter. Finite Coxeter Groups and Generalized Elnitsky Tilings. \texttt{arXiv:2111.10669}.
\bibitem{nicolaidesrowley2}Nicolaides, Robert; Rowley, Peter. Subtilings of Elnitsky's Tilings for Finite Irreducible Coxeter Groups. \texttt{arXiv 2106.03173}.
\bibitem{nicolaidesrowley3} Nicolaides, Robert; Rowley,  Peter. The Deletion Order and Coxeter Groups.
\texttt{arXiv:2407.07881}.
\bibitem{github} Nicolaides Robert; Rowley Peter.  Elnitsky Tiling Code Repository. \texttt{https://github.com/robnico/Elnitsky-Tilings-Mathematica.git}
\bibitem{tenner1}Tenner, Bridget Eileen. Enumerating in Coxeter Groups (Survey)
Advances in Mathematical Sciences, Association for Women in Mathematics Series, Springer (2020), no. 21, 75-82.
\bibitem{tenner2} Tenner, Bridget Eileen. Tiling-based models of perimeter and area. Adv. in Appl. Math. 119 (2020), no. 29 
\end{thebibliography}
\end{document}